\documentclass{amsart}
\usepackage[dvips]{color}
\usepackage{graphicx}

\usepackage{latexsym}
\usepackage{amssymb}
\usepackage{amsmath}
\usepackage{amsthm}
\usepackage{amscd}
\theoremstyle{plain}
\newtheorem{prop}{Proposition}
\newtheorem{thm}{Theorem}
\newtheorem{lem}{Lemma}
\newtheorem{Bob}{Definition}

\newtheorem{cor}{Corollary}
\theoremstyle{remark}

\newtheorem{rem}{Remark}

\newtheorem{ques}{Question}

\newcommand{\LS}{\ensuremath{\underset{n=1}{\overset{\infty}{\cap}} \, {\underset{i=n}{\overset{\infty}{\cup}}}\,}}

\begin{document}
\author[J.\ Chaika]{Jon Chaika}
\email{jonchaika@math.uchicago.edu}
\title[Hausdorff Dimension for IETs]{On the Hausdorff dimensions of a singular ergodic measure for some minimal interval exchange transformations}
\maketitle
We are interested in the Hausdorff dimension of ergodic measures for IETs. We provide a more complete description of phenomena than in \cite{me}. We briefly recall the question this paper addresses here.
\begin{Bob}
Let $\Delta_{n-1}=\{(l_1,...,l_n):l_i>0, l_1+...+l_n=1 \} $ be the $(n-1)$-dimensional simplex. Given $L=(l_1,l_2,...,l_n)\in \Delta_{n-1}$
 we can obtain n subintervals of the unit
interval: $I_1=[0,l_1) ,
I_2=[l_1,l_1+l_2),...,I_n=[l_1+...+l_{n-1},1)$. If we are also given
 a permutation on n letters $\pi$ we obtain an n-\emph{Interval Exchange Transformation} (IET)  $ T_{\pi,L} \colon [0,1) \to
 [0,1)$ which exchanges the intervals $I_i$ according to $\pi$. That is, if $x \in I_j$ then \begin{center} $T_{\pi,L}(x)= x - \underset{k<j}{\sum} l_k +\underset{\pi(k')<\pi(j)}{\sum} l_{k'}$.\end{center}
\end{Bob}
Interval exchange transformations can be minimal but not uniquely ergodic. Let us consider a minimal (that is, every orbit is dense) interval exchange $T_{\pi,L}$ with ergodic measures $\mu_1$ and $\mu_2$. Let $$L_c=\left(c\mu_1(I_1(T))+(1-c)\mu_2(I_1(T)),...,c\mu(I_d(T))+(1-c)\mu_2(I_d(T))\right).$$  
The IET $S_{\pi,L_c}$ is also minimal and not uniquely ergodic. When $c \in (0,1)$ Lebesgue measure is a preserved but not ergodic measure. When $c \in \{0,1\}$ Lebesgue measure is ergodic and there is another singular ergodic measure. See \cite[Section 1]{ietv} for a more general discussion. In this setting one can ask what is the Hausdorff dimension of the singular ergodic measure. This is equivalent to creating two new metrics on $[0,1)$, $d_{\mu_i}(a,b)=\mu_i([a,b])$ and asking what is the Hausdorff dimension of $\mu_1$ with respect to the metric $d_{\mu_2}$ and vice-versa.

Michael Keane introduced a construction of a minimal but not uniquely ergodic 4-IET \cite{nonue}. This construction is based on proving that there are orbits that have asymptotically different distribution.  
 It leads to two different ergodic measure $\lambda_2$ and $\lambda_3$ (see section \ref{measure Keane}). We use Keane's construction to show  results on the possible size of ergodic measures in terms of Hausdorff dimension. The main results of this paper are: 
\begin{thm}\label{onto} 
(a)$H_{dim}(\lambda_2,d_{\lambda_3})$ can take any value in $[0,1]$.

(b) $H_{dim}(\lambda_3,d_{\lambda_2})$ can take any value in $[0,1].$
\end{thm}
\begin{thm} \label{flip} $(H_{dim}(\lambda_2,d_{\lambda_3}), H_{dim}(\lambda_3,d_{\lambda_2}))$ can take values $(0,0),(1,0),(0,1) $ or $(1,1)$.
\end{thm}
\begin{Bob} Given $T \colon [0,1) \to [0,1)$, a $\mu$-ergodic map, we say a point $x_0\in [0,1)$ is \emph{generic for} $\mu$ if $\underset{N \to \infty}{\lim} \frac 1 {N} \underset{n=1}{\overset{N}{\sum}}f(T^n(x_0))=\int_0^1 fd\mu$ for every $f \in C([0,1])$.
\end{Bob}
The definition requires that the limit exists.
\begin{thm} \label{generic} There exists a minimal non uniquely ergodic IET $T$  where the complement of Lebesgue generic points has Hausdorff dimension 0. 
\end{thm}
Lebesgue measure is ergodic in this example. This says that all but a set of Hausdorff dimension zero of the points behave Lebesgue typically. Recall that a dense G$_{\delta}$ set of points are not generic for any ergodic measure of a continuous, not uniquely ergodic, minimal map of a compact metric space. On the other hand, by the Birkhoff Ergodic Theorem (and the fact that $C[0,1]$ with supremum norm has a countable dense set) if $\mu$ is an ergodic probability measure then $\mu$ almost every point is $\mu$ generic.

The first section provides a description of Keane's construction. The second section proves bounds on the measures of subintervals. The third section briefly recalls Hausdorff dimension and proves the theorems. Some concluding remarks are made at the end of the paper. There is an appendix that shows that the two ergodic measures can approximate each other differently.

\section{An introduction to Keane type examples}

Consider IETs with permutation $(4213)$. Observe that the second interval gets shifted by $l_4-l_1$. If this difference is small relative to $l_2$ then much of $I_2$ gets sent to itself. At the same time, pieces of $I_3$ do not reach $I_2$ until they have first reached $I_4$. This is the heart of the Keane construction. The details of the Keane construction are centered  around iterating this procedure by the first return map. 
\begin{Bob} Let $T:[0,1) \to [0,1)$ be a Lebesgue measure preserving transformation and $J\subset [0,1)$. $T|_J\colon J \to J$ denotes the \emph{first return map} to $J$. That is if $x \in J$ let $r(x) \min\{n>0: T^n(x) \in J\}$. $T|_J(x)=T^{r(x)}(x)$.
\end{Bob}
Keane considered the first return map on the fourth interval, which we denote $I^{(1)}$. The first return map on this interval is once again a 4-IET. (The induced map of an IET on $I_j$ is an IET on at most the same number of intervals. This is in general false for the induced map of an IET on $[a,b)$.) 
Keane showed that by choosing the lengths appropriately one could ensure that this induced map had the permutation (2431). Name these in reverse order and we once again get a (4213) IET. Motivated by this, we name the 4 exchanged subintervals of $I^{(1)}$ under $T|_{I^{(1)}}$ in reverse order; that is, $I^{(1)}_1$ is the subinterval furthest to the right. Keane also showed that for any choice $ m,n \in \mathbb{N} $ one can find
 an IET whose landing pattern of $I^{(1)}_j$ is given by the columns of following matrix:
\begin{center} $A_{m,n}=
\left(
\begin{array}{cccc}
0 & 0 & 1 &1 \\
m-1 & m & 0 & 0 \\
n & n & n-1 & n \\
1 & 1 & 1 & 1 \end{array} \right)$; $\quad$ $m, n \in \mathbb{N} =\{1, 2, ...\} $.
\end{center} 

 In order to see this, pick lengths for $I^{(1)}$ and write it as a column vector. Now assign lengths to the original IET by multiplying this column vector by $A_{m,n}$. The induced map will travel according to this matrix by construction. For instance, if one chooses lengths $[\frac 1 4,\frac 1 4,\frac 1 4,\frac 1 4]$ for $I^{(1)}$ one gets lengths of 
$$[\frac{2}{2+2m-1+4n-1+4}, \frac{2m-1}{2m+4n+4}, \frac{4n-1}{2m+4n+4}, \frac {4}{2m+4n+4}]$$ for the original IET (after renormalizing).
For any finite collection of matrices one can iterate this construction. (Assign lengths for $I^{(k)}$ by multiplying the lengths of $I^{(k+1)}$ by $A_{m_{k+1},n_{k+1}}$, multiply the resulting column vector by $A_{m_{k},n_{k}}$ and so forth. $I^{(k+1)}$ is defined inductively as the fourth interval of $I^{(k)}$.) Compactness (of $\Delta_3$, which can be thought of as the parameterizing space of (4213) IETs) ensures that we can pass to an infinite sequence of these matrices.

Since the intervals are named in reverse order, the discontinuity (under the induced map) between $I_2^{(1)}$ and $I_3^{(1)}$ is given by $T^{-1}(\delta_1)$ where $\delta_1$ denotes the discontinuity between $I_1$ and $I_2$. As the first row of the matrix suggests ${I_1=T(I_4^{(1)} \cup I_3^{(1)})}$. 
The discontinuity (under the induced map) between $I_1^{(1)}$ and $I_2^{(1)}$ is given by $T^{-m}(\delta_2)$ where $\delta_2$ denotes the discontinuity between $I_2$ and $I_3$. As the second row of the matrix suggests $$I_2=T(I_2^{(1)} \cup I_1^{(1)})\cup T^2(I_2^{(1)} \cup I_1^{(1)}) \cup ...\cup T^{m-1}(I_2^{(1)} \cup I_1^{(1)}) \cup T^m (I_2).$$ 
The discontinuity (under the induced map) between $I_3^{(1)}$ and $I_4^{(1)}$ is given by $T^{-n-1}(\delta_3)$ where $\delta_3$ denotes the discontinuity between $I_3$ and $I_4$.
 As the third row of the matrix suggests $$I_3= T^m(I_1^{(1)}) \cup T^{m+1}(I_2^{(1)}) \cup T^2(I_4^{(1)} \cup I_3^{(1)}) \cup T^{m+1}(I_1^{(1)}) \cup T^{m+2}(I_2^{(1)}) \cup T^3(I_4^{(1)} \cup I_3^{(1)}) \cup$$
 $$... \cup T^{m+n-1}(I_1^{(1)}) \cup T^{m+n}(I_2^{(1)}) \cup T^n(I_4^{(1)} \cup I_3^{(1)}) \cup T^{m+n}(I_1^{(1)}) \cup T^{m+n+1}(I_2^{(1)}) \cup T^{n+1}(I_4^{(1)}).$$  $I_4=I_4^{(1)} \cup I_3^{(1)} \cup I_2^{(1)} \cup I_1^{(1)}$. 
As the columns of the matrix suggest, this is also $$I_4= T^{n+1}(I_3^{(1)}) \cup T^{m+n+1}(I_2^{(1)}) \cup T^{m+n}(I_1^{(1)}) \cup T^{n+2}(I_4^{(1)}).$$
 To summarize, the composition of $I_j$ can be given by the $j^{th}$ row of the matrix. The travel before first return of $I_j^{(1)}$ can be given by the $j^{th}$ column. Additionally, because the intervals were named in reverse order, the permutation of the induced map is once again $(4213)$.

It is important for this construction that everything be iterated. The composition of $I_j^{(k)}$ in pieces of $I^{(k+r)}$ is given by $e_j^{\tau} A_{m_{k+1},n_{k+1}}...A_{m_{k+r},n_{k+r}}$ (where $e_j^{\tau}$ denotes the transpose pf $e_j$). Likewise, the travel of $I_j^{(k+r)}$ under $T|_{I^{(k)}}$ before first return to $I^{(k+r)}$ is given by $A_{m_{k+1},n_{k+1}}...A_{m_{k+r},n_{k+r}}e_j$.

\begin{Bob}Let $O(I_j^{(k)})$ denote the disjoint images under $T$ of $I_j^{(k)}$ before first return to $I^{(k)}$.
\end{Bob}
\begin{Bob} Let $b_{k,i}$ denote the first return time of $I_i^{(k)}$ to $I^{(k)}$. 
\end{Bob}
\begin{rem} $b_{k,i}$ is given by $|A_{m_1,n_1}...A_{m_k,n_k}e_i|_1 $. In particular, $b_{k,2}=m_{k}b_{k-1,2}+n_{k}b_{k-1,3}+b_{k-1,4}$ and $b_{k,3}=b_{k-1,1}+(n_{k}-1)b_{k-1,3}+b_{k-1,4}$.
\end{rem}
\begin{rem} $O(I_i^{(k)})= \underset{i=1}{\overset{b_{k,i}-1}{\cup}}T^i(I_j^{(k)})$.
\end{rem}

Now for some explicit statements about the travel of subintervals of $I^{(k)}$ under the induced map $T|_{I^{(k)}}$. When $I_3^{(k)}$ returns to $I^{(k)}$ it entirely covers $I_4^{(k)}$. It is a subset of $I_3^{(k)} \cup I_4^{(k)}$.
When $I_4^{(k)}$ returns to $I^{(k)}$ it entirely covers $I_1^{(k)}$. It intersects $I_2^{(k)}$. Moreover part of this intersection will stay in $O(I_2^{(k)})$ for the next $m_{k+1}b_{k,2}$ images (the other part $(m_{k+1}-1)b_{k,2}$.)
When $I_2^{(k)}$ returns to $I^{(k)}$ it intersects $I_3^{(k)}$. Moreover this piece of intersection will stay in $O(I_3^{(k)})$ for the next $n_{k+1}b_{k,3}$ images.


Some facts to keep in mind:
\begin{enumerate}
\item The choice of $n_k$ has no effect on $b_{i,2}$ for $i <k$.
\item The choice of $n_k$ has no effect on $b_{i,3}$ for $i < k$.
\item The choice of $m_k$ has no effect on $b_{i,2}$ for $i < k$.
\item The choice of $m_k$ has no effect on $b_{i,3}$ for $i < k+1$.
\end{enumerate}
\section{Measure estimates for Keane's construction}\label{measure Keane}

The previous section discussed the \emph{topological} properties of Keane type IETs. Keane's construction of these IETs was motivated by their measure properties.

 In Keane's example we have a minimal non-uniquely ergodic  4-IET $T$ with ergodic measure $\lambda_2$ and $\lambda_3$. To gain some further intuition consider the product:
\begin{center} $
\left(
\begin{array}{cccc}
0 & 0 & 1 &1 \\
m-1 & m & 0 & 0 \\
n & n & n-1 & n \\
1 & 1 & 1 & 1 \end{array} \right)  \left( \begin{array}{c} a \\ b \\ c \\ d \end{array} \right)=\left( \begin{array}{c} c+d \\ (m-1)a+mb \\ n(a+b+ c +d)-c\\ a+b+c+d \end{array} \right)$
\end{center} 
Notice that if $a=c=d=0$, $b=1$, $m$ is much bigger than $n$ and $m$ is large then the resulting column vector has small angle with the original. Likewise, if  $a=b=d=0$, $c=1$ and $n$ is large then the resulting column vector has small angle with the original. Motivated by this, we introduce another piece of notation.
\begin{Bob} Let $\bar{A}_{m,n}v= \frac{A_{m,n}v}{|A_{m,n}v|}$, where $|w|$ is the sum of the entries in $w$.
\end{Bob} 
 Michael Keane proved:
\begin{thm} (Keane \cite{nonue}) If $3(n_k+1)\leq m_k\leq \frac 1 2 (n_{k+1}+1)$ and $n_1>9$ then an IET with lengths given by 
$$\underset{r \to \infty}{\lim}\underset{k=1}{\overset{r}{\prod}} \bar{A}_{n_k,m_k}v$$
 is minimal but not uniquely ergodic for any $v \in \Delta_3$.
  Moreover it has two ergodic measure $\lambda_2$ and $\lambda_3$
   which assign measures to intervals given by $$\underset{r \to \infty}{\lim}\underset{k=1}{\overset{r}{\prod}} \bar{A}_{n_k,m_k}e_2$$
    and $$\underset{r \to \infty}{\lim}\underset{k=1}{\overset{r}{\prod}} \bar{A}_{n_k,m_k}e_3$$ respectively. 
\end{thm}
 In particular he showed the limit exists. 
One can remove the assumption on $n_1$ or any finite number of matrices in Keane's Theorem.

\subsection{Estimates on the size of intervals with respect to the two ergodic measures} \label{estimates}

In this section we bound $\lambda_i(I_j^{(k)})$ between two constants. Many of these are needed in the later arguments. We include the rest for completeness.

In these computations, we use $j$th entry of partial products $\underset{t=1}{\overset{r}{\prod}}\bar{A}_{k+t}e_i$ to estimate $\frac{\lambda_i(I_j^{(k)})}{\lambda_i(I^{(k)})}$. To complete these estimates we remark that $b_{k,2}^{-1}>\lambda_2(I^{(k)})>(4b_{k,2})^{-1}$ (Lemma \ref{L2bigorbit}) and $b_{k,3}^{-1}>\lambda_3(I^{(k)})>(8b_{k,3})^{-1} $ (Lemma \ref{L3bigorbit}).
\begin{rem} The proofs of these lemmas often provide better results than their statements. Additionally, it is often straightforward to provide better estimates, especially under stronger growth conditions on $m_i$ and $n_i$. Lemma \ref{L2I2big}, for instance, would be amenable to such an approach.
\end{rem}
\begin{prop} $\frac{\lambda_i(I_j^{(k)})}{\lambda_i(I^{(k)})}=$ the $j$th entry of $ \underset{r \to \infty}{\lim}\underset{t=1}{\overset{r}{\prod}}\bar{A}_{m_{k+t},n_{k+t}}e_i$.
\end{prop}
\begin{lem}\label{L3I2big} $\frac{\lambda_3(I_2^{(k)})}{\lambda_3(I^{(k)})} \geq \frac {m_{k+1}}{2n_{k+1} n_{k+2}}$.
\end{lem}
\begin{proof} It suffices to show that the second entry of $\bar{A}_{m_{k+1},n_{k+1}}\bar{A}_{m_{k+2},n_{k+2}}e_3>\frac {m_{k+1}}{2n_{k+1} n_{k+2}}$. 
This is a direct computation.
\end{proof}
\begin{lem} \label{L3I2small} $\frac{\lambda_3(I_2^{(k)})}{\lambda_3(I^{(k)})} \leq \frac {2m_{k+1}} {(n_{k+2}+1)(n_{k+1} +1)}$.
\end{lem}
This result is in the proof of Lemma 3 of \cite{nonue}.
\begin{lem} \label{L3I3big} $\frac{\lambda_3(I_3^{(k)})}{\lambda_3(I^{(k)})} \geq 1- \frac {3} {n_{k+1}}$.
\end{lem}
This is Lemma 3 of \cite{nonue}.
\begin{lem} \label{L3I4small} $\frac{\lambda_3(I_4^{(k)})}{\lambda_3(I^{(k)})} \leq \frac {1} {n_{k+1}}$.
\end{lem}
\begin{proof} 
Notice that $I_4^{(k)}$ is the disjoint union of an image of $I_1^{(k+1)}$, an image of $I_2^{(k+1)}$, an image of $I_3^{(k+1)}$ and an image of $I_4^{(k+1)}$ and that $I^{(k)}$ contains at least $n_{k+1}+1$ disjoint images of $I_j^{(k+1)}$ for each $j$.
\end{proof}
\begin{lem}\label{L3I4big} $\frac{\lambda_3(I_4^{(k)})}{\lambda_3(I^{(k)})} \geq \frac {1} {2n_{k+1}}$.
\end{lem}
\begin{proof} $I_4^{(k)}$ is made up of one disjoint image of each $I_i^{(k+1)}$. $I_3^{(k)}$ is made up of $n_{k+1}-1$ disjoint images of $I_3^{(k+1)}$ and $n_{k+1}$ disjoint images of each of the other $I_i^{(k+1)}$. Therefore, because $n_{k+1}$ images of $I_4^{(k+1)}$ cover $I_3^{(k+1)}$ and ${\frac {\lambda_3(I_4^{(k)})}{\lambda_3(I^{(k)})}> \frac{\lambda_3(I_3^{(k)})}{\lambda_3(I^{(k)})} \frac 1 {n_{k+1}}}$. The lemma follows by Lemma \ref{L3I3big}.
\end{proof}
\begin{lem} \label {L3I1small}$\frac{\lambda_3(I_1^{(k)})}{\lambda_3(I^{(k)})} \leq \frac {1} {n_{k+1}}$.
\end{lem}
\begin{proof} $I_1^{(k)}$ is made up of a disjoint union of an image of $I_3^{(k+1)}$ and $I_4^{(k+1)}$ each of which has at least $n_{k+1}+1$ disjoint images in $I^{(k)}$.
\end{proof}
\begin{lem}\label{L3I1big} $\frac{\lambda_3(I_1^{(k)})}{\lambda_3(I^{(k)})} \geq \frac {1} {3n_{k+1}}$.
\end{lem}

\begin{proof} It follows from the composition of $I_i^{(k)}$ by subintervals of $I^{(k+1)}$ that $\lambda_3(I_1^{(k)}) \geq \lambda_3(I_3^{(k+1)})$. The proof follows from Lemmas \ref{L3I4big} and \ref{L3I3big}.
\end{proof}
\begin{lem}\label{L2I2big} $\frac {\lambda_2(I_2^{(k)})}{\lambda_2(I^{(k)})}> \frac { m_{k+1}}{4(n_{k+1}+m_{k+1}+2)}$.
\end{lem}
\begin{proof} Observe that if $v \in \mathbb{R}_+^4$ is positive, $|v|_1=1$ and $v[2]>\frac 1 4$ then $\bar{A}_{m,n}v[2]>\frac 1 4 $ so long as $m\geq 3n$ and $n>\frac 8 5$. By induction, it follows that ${\underset{t=k+1}{\overset{r}{\prod}} \bar{A}_{m_t,n_t} e_2 [2]>\frac { m_{k+1}}{4(n_{k+1}+m_{k+1}+2)}}$.
\end{proof}
\begin{lem}\label{L2I3small} $\frac{\lambda_2(I_3^{(k)})}{\lambda_2(I^{(k)})}\leq  \frac{4n_{k+1}}{m_{k+1}}$.
\end{lem}
\begin{proof} By the previous proof, $\underset{t=k+2}{\overset{r}{\prod}} \bar{A}_{m_t,n_t} e_2 [2]>\frac {1}{4}$. Therefore $A_{m_{k+1},n_{k+1}}\underset{t=k+2}{\overset{r}{\prod}} \bar{A}_{m_t,n_t} e_2 [2]>\frac {m_{k+1}}{4}$. Observing that $A_{m_{k+1},n_{k+1}}v[3]\leq n_{k+1}|v|$ for any $v \in \mathbb{R}^4_+$ implies that $\underset{t=k+1}{\overset{r}{\prod}} \bar{A}_{m_t,n_t} e_2 [3]<\frac {4n_{k+1}} { m_{k+1}}$.
\end{proof}
Before the next estimate we need a lemma.
\begin{lem} \label{b2 bigger} $b_{k,2}>b_{k,i}$ for $i \in \{1,3,4\}$.
\end{lem}
\begin{proof} Notice that $b_{k,2}>b_{k,1}$ because the
 second entry of $A_{m_k, n_k} e_2=m_k > m_k-1$ and $m_k-1$ is the second entry of $A_{m_k, n_k} e_1$.
$A_{m_k,n_k}e_2$ agrees with
$A_{m_k,n_k}e_1$ in all other entries. Also, $b_{k,2} \geq b_{k,j}$ for $j=3,4$ because $A_{m_k,n_k}e_2 \geq A_{m_k,n_k} e_j$ in
all entries but the first and
$m_kA_{m_{k-1},n_{k-1}}e_2>A_{m_{k-1},n_{k-1}}e_1$ in all entries (the second entry of $A_{m_k,n_k}e_j$ is 0 and the second entry
of $A_{m_k,n_k}e_2$ is $m_ke_2$ and also the first entry of $A_{m_k,n_k}e_j =1$). This argument shows that $A_{m_{k-1},n_{k-1}}A_{m_k,n_k}e_2$ has each entry greater than or equal to the corresponding entries of $A_{m_{k-1},n_{k-1}}A_{m_k,n_k}e_j$ for $j=3,4$.
\end{proof}
\begin{lem}\label{L2I3big} $\frac {\lambda_2(I_3^{(k)})} {\lambda_2(I^{(k)})} \geq \frac {n_{k+1}} {2m_{k+1}}$.
\end{lem}
\begin{proof} By inspection $\bar{A}_{m_{k+1},n_{k+1}}e_2[3]=\frac{n_{k+1}}{m_{k+1}+n_{k+1}+1}>\frac {n_{k+1}} {2m_{k+1}}$. We now prove  $\bar{A}_{m_{k+1},n_{k+1}}e_2[3]<\bar{A}_{m_{k+1},n_{k+1}}e_i[3]$ for $i=1,3,4$. This is because $|A_{m_{k+1},n_{k+1}}e_2|>|A_{m_{k+1},n_{k+1}}|$ for $i=1,3,4$ (Lemma \ref{b2 bigger}) and $A_{m_{k+1},n_{k+1}}e_i=n_{k+1}$ for $i=1,2,4$. For $i=3$ notice that $|A_{m_{k+1},n_{k+1}}e_2|>3|A_{m_{k+1},n_{k+1}}e_3|$ and $A_{m_{k+1},n_{k+1}}e_3=n_{k+1}-1>\frac {n_{k+1}}{3}$. Thus $\bar{A}_{m_{k+1},n_{k+1}}(\bar{A}_{m_{k+2},n_{k+2}}...\bar{A}_{m_{k+r},n_{k+r}})e_2[3]\geq\frac {n_{k+1}} {2m_{k+1}}$.
\end{proof}
\begin{lem}\label{L2I4big} $\frac {\lambda_2(I_4^{(k)})}{\lambda_2(I^{(k)})}> \frac 1 {2m_{k+1}}$.
\end{lem}
\begin{proof} There are at most $m_{k+1}+n_{k+1}+1$ disjoint images of any $I_i^{(k+1)}$ in $I^{(k)}$. By our standard assumptions $n_{k+1}+1<m_{k+1}$. Also $I_4^{(k+1)}$ is made up of one image of each $I_i^{(k+1)}$. Thus $2m_{k+1}$ copies of $I_4^{(k)}$ cover $I^{(k)}$. 
\end{proof}
\begin{lem} \label{L2I4small} $\frac {\lambda_2(I_4^{(k)})}{\lambda_2(I^{(k)})}<\frac 4 {m_{k+1}}$.
\end{lem}
\begin{proof} By construction the fourth entry of $A_{m_{k+1},n_{k+1}}(\bar{A}_{m_{k+2},n_{k+2}}...\bar{A}_{m_{k+r},n_{k+r}})$ is 1. By Lemma \ref{L2I2big} the second entry is at least $\frac 1 4 m_{k+1}$.
\end{proof}
\begin{lem} \label{L2I1small} $\frac {\lambda_2(I_1^{(k)})}{\lambda_2(I^{(k)})}< \frac {16n_{k+2}+ 16} {m_{k+1}m_{k+2}}$.
\end{lem}
\begin{proof} $I_1^{(k)}$ is made up of one image of $I_3^{(k+1)}$ and one image of $I_4^{(k+1)}$. ${\frac {\lambda_2(I_1^{(k)})}{\lambda_2(I^{(k)})}= \frac {\lambda_2(I^{(k+1)})}{\lambda_2(I^{(k)})} \frac {\lambda_2(I_3^{(k+1)} \cup I_4^{(k+1)})}{\lambda_2(I^{(k+1)})}}$. By the fact that $I^{(k+1)}=I_4^{(k)}$, Lemma \ref{L2I3small} and Lemma \ref{L2I4small} this is less than $\frac 4 {m_{k+1}} \frac {4n_{k+2}+4}{m_{k+2}}.$
\end{proof}
\begin{lem} \label{L2I1big} $\frac {\lambda_2(I_1^{(k)})}{\lambda_2(I^{(k)})}> \frac {n_{k+2}} {4m_{k+1}m_{k+2}}$.
\end{lem}
\begin{proof} $I_1^{(k)}$ contains one image of $I_3^{(k+1)}$. 
By Lemma \ref{L2I3big}, $\frac {\lambda_2(I_3^{(k+1)})}{\lambda_2(I^{(k+1)})}> \frac {n_{k+2}}{2m_{k+2}}$ and by Lemma \ref{L2I4big}, $\frac{\lambda_2(I^{(k+1)})}{\lambda_2(I^{(k)})}> \frac 1 {2m_{k+1}}$.
\end{proof}

\section{Hausdorff dimension for ergodic measures in Keane type examples}
\subsection{Definition of Hausdorff dimension} Given a metric $D$
let $diam(U)=\underset{ x , y \in U}{\sup}D(x,y)$.
Consider a set $S \subset [0,1)$. We say a collection of open sets $\mathcal{U}=\{U_i\}_{i=1}^{\infty}$ is a $\delta>0$ cover of $S$ if $S \subset \underset{i=1}{\overset{\infty}{\cup}} U_i$ and $diam(U_i)\leq \delta$ $\forall i$. Let $H_{\delta}^s(S)=\inf \{\underset{i=1}{\overset{\infty}{\sum}} |U_i|^s: \{U_i\} \text{ is a } \delta \text{ cover of } S\} $. Let $H^s(S)=\underset{\delta \to 0^+}{\lim}H_{\delta}^s(S)$. Notice that the limit exists. Let $H_{dim}(S)=\inf \{s: H^s(S)=0\}$. This is equivalent to defining $H_{dim}(S)=\sup \{s:H^s(S)=\infty\}$. We state a few well known properties of Hausdorff dimension.

$H_{dim}(\underset{i=1}{\overset{\infty}{\cup}} S_i)=\underset{i}{\sup}\, H_{dim}(S_i)$.

$H_{dim} (\underset{i=1}{\overset{\infty}{\cap}} S_i) \leq \underset{i}{\inf}\, H_{dim}(S_i)$. 

\begin{Bob}
For a Borel Measure $\mu$ we define the \emph{Hausdorff dimension of a probability measure
$\mu$} is \begin{center}$H_{\text{dim}}(\mu)=
\inf\{H_{\text{dim}}(M) \colon M \text{ is Borel and } \mu(M)=1 \}$.\end{center}
\end{Bob}

For upper bounds to Hausdorff dimension of a set, explicit coverings are often all that is necessary. For lower bounds Frostman's Lemma is useful.

\begin{lem} (Frostman) Let $B \subset [0,1)$  be a Borel set. $H^s(B)>0$ iff there exists a finite radon measure on $B$, $\nu$,  such that for all $x$ and $r>0$ we have $\nu(B(x,r)) \leq r^s$.
\end{lem}
See \cite[p. 112]{mattila}.
\begin{cor} \label{keanefrost} If $\mu$ is a measure on $[0,1)$ and $\epsilon_1,...$ is a positive sequence tending to 0 such that $\frac {\epsilon_i}{\epsilon_{i+1}}<C$ for some $C$ and all $i$ then $\mu(B(x,\epsilon_i))<C(\epsilon_i)^{\alpha}$ implies $H_{dim}(\mu) \geq \alpha$.
\end{cor}
\begin{lem} If $T$ is a piecewise isometry then $H_{dim}(T(S))\leq H_{dim}(S)$.
\end{lem}
This holds for locally Lipshitz maps as well, but this fact is unnecessary for the present paper.

\subsection{Estimates towards calculating the Hausdorff dimension for ergodic measures of IETs}
For upper bounds to the Hausdorff dimension for an ergodic measure of an IET the following proposition is useful.
\begin{prop} Let $T$ be a $\mu$-ergodic IET and the $H_{dim}(\mu)=t$. If $S$ is a set such that $H_{dim}(S)<t$ then $\mu(S)=0$.
\end{prop}
\begin{proof}
This follows from the countable stability of $H_{dim}$ and ergodicity. If $\mu(S)>0$ then $\mu(\underset{i=1}{\overset{\infty}{\cup}}T^i(S))=1$ by ergodicity. However, by the countable stability of Hausdorff dimension $H_{dim}(\underset{i=1}{\overset{\infty}{\cup}}T^i(S))= H_{dim}(S)$ because $T$ is a piecewise isometry.
\end{proof}
This proposition says that one needs to only prove upper bounds on part of the measure. If $\mu(S)>0$ and $H^t(S)=0$ then $H_{dim}(\mu) \leq t$.

Below is a lemma based adapting Frostman's Lemma to our particular circumstances to provide lower bounds for the Hausdorff dimension of an ergodic measure.
\begin{lem} \label{loc est} If there exists $C$ such that $C\lambda_3(I_i^{(k)})^{\alpha}>\lambda_2(I_i^{(k)})$ for any $k$ and $i \in \{1,2,3,4\}$ then $H_{dim}(\lambda_2,d_{\lambda_3})\geq \alpha$. Likewise, if there exists a $C$ such that $C\lambda_2(I_i^{(k)})^{\alpha}>\lambda_3(I_i^{(k)})$ for any $k$ and $i \in \{1,2,3,4\}$ then $H_{dim}(\lambda_3,d_{\lambda_2})\geq \alpha$. 
\end{lem}
\begin{proof} By Frostman's Lemma it suffices to show that for any interval $J$ we have $C\lambda_3(J)^{\alpha}> \lambda_2(J)$.
 We will show that $\log_{\lambda_3(J)} \lambda_2(J)$ is dominated by something comparable to $\log_{\lambda_3(I_i^{(t)})} \lambda_2(I_i^{(t)})$.
 This follows from the fact that $I_2^{(k)}$ and $I_3^{(k)}$ are made up of repeating images.
  To see this assume that we wish to estimate  $\log_{\lambda_3(J)} \lambda_2(J)$ mostly covered by images of $I_i^{(k+1)}$ and contained in $I_2^{(k)}$. 
  $I_2^{(k)}$ is made up of repeating unions of images of $I_1^{(k+1)} \cup I_2^{(k+1)}$.
If $J'$ varies over intervals that are unions of $I_i^{(k+1)} $ contained in $I_2^{(k)}$ then $\log_{\lambda_3(J')} \lambda_2(J')$ is minimized by choosing $J'=I_2^{(k+1)}$, $I_2^{(k+1)} \cup I_1^{(k+1)} \cup I_2^{(k+1)}$ or $J'=I_2^{(k)}$.
   In either case, $\log_{\lambda_3(J)} \lambda_2(J)$ is dominated by something proportional to $I_2^{(j)}$ for some $j$. Likewise, if $J \subset I_3^{(k)}$ for pieces in images of $I_3^{(k)}$ one either covers by all of $I_3^{(k)}$ or $I_2^{(k+1)} \cup I_4^{(k+1)} \cup I_3^{(k+1)} \cup I_1^{(k+1)} \cup I_2^{(k+1)}$. $I_1^{(k)}$ and $I_4^{(k)}$ are made up of at most 1 image of each $I_i^{(k+1)}$ for $i \in \{1,2,3,4\}$ and so reduce to these cases. Similar arguments hold for $\log_{\lambda_2(J)} \lambda_3(J)$.
\end{proof}

\begin{lem}\label{b2 est} $b_{k,2}\leq \underset{i=1}{\overset{k}{\Pi}} 2m_i$.
\end{lem}
\begin{proof}Recall $b_{k,2}=m_kb_{k-1,2}+n_k b_{k-1,3} +b_{k-1,4}$. By Lemma \ref{b2 bigger} $b_{i,2} \geq b_{i,j}$. By our assumptions $m_i>n_i+1$. The lemma follows by induction.
\end{proof}
\begin{lem} \label{b3 est} $\underset{i=1}{\overset{k}{\prod}} n_{i} <b_{k,3}$.
\end{lem}
\begin{proof}Recall $b_{k,3}=b_{k,1}+(n_{k}-1)b_{k-1,3} +b_{k-1,4}$. Notice that $b_{i,4}=b_{i-1,1}+n_ib_{i-1,3}+b_{i-1,4}>b_{i,3}$ implying that $b_{k,3}>n_kb_{k-1,3}$. The lemma follows by induction.
\end{proof}
\begin{lem} \label{L3bigorbit} $\lambda_3 (O(I_3^{(k)}))>\frac 1 8$. 
\end{lem}
\begin{proof}For any $i$ we have $n_{k+1}b_{k,3}>\frac 1 2 b_{k,i}$. To complete the proof consider
 $\underset{i \neq 3}{\sum}b_{k,i} \frac{\lambda_3(I_i^{(k)})}{\lambda_3(I^{(k)})}< \frac{b_{k,1}}{n_{k+1}}+ \frac{b_{k,2}}{n_{k+1}}+\frac{b_{k,4}}{n_{k+1}}$
  while $b_{k,3}\frac{\lambda_3(I_3^{(k)})}{\lambda_3(I^{(k)})}$ is proportional to $b_{k,3}$.
\end{proof}
This Lemma establishes that $\lambda_3(I^{(k)})$ is proportional to $b_{k,3}^{-1}$.
\begin{lem}\label{L2bigorbit} $\lambda_2(O(I_2^{(k)}))>\frac 1 4$.
\end{lem}
\begin{proof} By Lemma \ref{b2 bigger} $b_{k,2}>b_{k,i}$  and so $\lambda_2(O(I_2^{(k)}))>\frac {\lambda_2(I_2^{(k)})}{\lambda_2(I^{(k)})}$.
\end{proof}
This Lemma establishes the $\lambda_2(I^{(k)})$ is proportional to $b_{k,2}^{-1}$.
\begin{prop} $H_{dim}(\lambda_2,d_{\lambda_3}) \leq H_{dim}(\LS O(I_2^{(k)}),d_{\lambda_3})$.
\end{prop}
\begin{proof} $\LS O(I_2^{(k)})$ has positive $\lambda_2$ measure and is $T$ invariant except for a set of measure zero (because $\lambda_2(I_2^{(k)}) \to 0$). By ergodicity it has full measure.
\end{proof}
\begin{prop} \label{L2 up est} $H_{dim}(\lambda_2,d_{\lambda_3}) \leq \underset{k \to \infty}{\liminf} \log_{\lambda_3(I_2^{(k)})} b_{k,2}^{-1}$.
\end{prop}
\begin{proof} Assume that $\underset{k \to \infty}{\liminf} \log_{\lambda_3(I_2^{(k)})} b_{k,2}^{-1}=s$. It suffices to show that $H_{dim}(\lambda_2,d_{\lambda_3})<s+\epsilon$ for all $\epsilon>0$. Let $k_1,k_2,...$ be an increasing sequence of natural numbers such that $\log_{\lambda_3(I_2^{(k_t})} b_{k_t,2}^{-1}<s+\epsilon$ for all $t$. Consider $\LS O(I_2^{(k_i)})$. It has positive  $\lambda_2$ measure by Lemma \ref{L2bigorbit}. The naive covering shows that $H^{s+\epsilon}(\LS O(I_2^{(k_i)}))=0$. That is, fix $\delta>0$ and choose $i$ such that $\lambda_3(I_2^{(k_2)})<\delta$. We bound $ H^{s+\epsilon}_{\delta}( \LS O(I_2^{(k_i)}))$ by covering each $O(I_2^{(k_i)})$ by $b_{k_i,2}$ images of $I_2^{(k_i)}$. By the fact that $\log_{\lambda_3(I_2^{(k_i)})} b_{k_i,2}^{-1}<s+\epsilon$ for all $i$  it follows that $\underset{i=1}{\overset{\infty}{\sum}} b_{k_i,2}(\lambda_3(I_2^{(k_i)}))^{s+2\epsilon}<\infty$ and we see that $H_{dim}(\lambda_2,d_{\lambda_3})<s+2\epsilon$ for any $\epsilon>0$.
\end{proof}
\begin{prop}\label{L2 low est} $ H_{dim}(\lambda_2,d_{\lambda_3}) \geq   \underset{k \to \infty}{\liminf} \log_{\lambda_3(I_2^{(k)})} (\lambda_2(I_2^{(k)}))$.
\end{prop}
\begin{proof}  By Lemma \ref{loc est} we have that  $$ H_{dim}(\lambda_2,d_{\lambda_3}) \geq \underset{1\leq i \leq 4}{\min}\underset{k \to \infty}{\liminf} \log_{\lambda_3(I_i^{(k)})} (\lambda_2(I_i^{(k)})).$$  Consider
$$\log_{\frac{\lambda_3(I_i^{(k)})}{\lambda_3(I^{(k)})} \lambda_3(I^{(k)})} \frac{\lambda_2(I_i^{(k)})}{\lambda_2(I^{(k)})} \lambda_2(I^{(k)}).$$ 
To determine the $i$ that attains the minimum it suffices to consider $\log_{\frac{\lambda_3(I_i^{(k)})}{\lambda_3(I^{(k)})}} \frac{\lambda_2(I_i^{(k)})}{\lambda_2(I^{(k)})}$. For all large $k$ the smallest of these is  $\log_{\frac{\lambda_3(I_2^{(k)})}{\lambda_3(I^{(k)})}} \frac{\lambda_2(I_2^{(k)})}{\lambda_2(I^{(k)})}<\ \log_{\frac {2m_{k+1}}{n_{k+1}n_{k+2}}} \frac 1 4 $ (see Section \ref{estimates}).
\end{proof}
\begin{prop} \label{L3 up est} $H_{dim}(\lambda_3,d_{\lambda_2}) \leq \underset{k \to \infty}{\liminf} \log_{\lambda_2(I_3^{(k)})} b_{k,3}^{-1}$.
\end{prop}
The proof is similar to Proposition \ref{L2 up est}.

\begin{prop} \label{L3 low est} $ H_{dim}(\lambda_3,d_{\lambda_2}) \geq   \underset{k \to \infty}{\liminf} \log_{\lambda_2(I_3^{(k)})} (\lambda_3(I_3^{(k)}))$.
\end{prop}
\begin{proof} By Lemma \ref{loc est} we have that  $$ H_{dim}(\lambda_3,d_{\lambda_2}) \geq \underset{1\leq i \leq 4}{\min}\underset{k \to \infty}{\liminf} \log_{\lambda_2(I_i^{(k)})} (\lambda_3(I_i^{(k)})).$$ 
Consider $$\log_{\frac{\lambda_2(I_i^{(k)})}{\lambda_2(I^{(k)})} \lambda_2(I^{(k)})} \frac {\lambda_3(I_i^{(k)})}{\lambda_3(I^{(k)})} \lambda_3(I^{(k)}).$$ To determine the $i$ that attains the minimum it suffices to consider $\log_{\frac{\lambda_2(I_i^{(k)})}{\lambda_2(I^{(k)})}} (\frac {\lambda_3(I_i^{(k)})}{\lambda_3(I^{(k)})}) $. The smallest of these is $\log_{\frac{\lambda_2(I_3^{(k)})}{\lambda_2(I^{(k)})}} (\frac {\lambda_3(I_3^{(k)})}{\lambda_3(I^{(k)})})<\log_{\frac{n_{k+1}}{2m_{k+1}}} (1-\frac 3 {n_{k+1}})$ (see Section \ref{estimates}).
\end{proof}


\subsection{Proofs of Theorems}
\begin{proof}[Proof of Theorem \ref{flip}] Choosing $m_{k}=n_k^k$ implies that $H_{dim}(\lambda_3, d_{\lambda_2})=0$ by Proposition \ref{L3 up est}. Likewise, choosing $n_{k+1}=m_k^k$ implies that $H_{dim}(\lambda_2,d_{\lambda_3})=0$ by Proposition \ref{L2 up est}. Choosing $m_k=4n_k$ implies that $H_{dim}(\lambda_3, d_{\lambda_2})=1$ by Proposition \ref{L3 low est}. Lastly, choosing $n_{k+1}=4m_k$ implies that $H_{dim}(\lambda_2, d_{\lambda_3})=1$ by Proposition \ref{L2 low est}. By suitable choices of $m_k$ and $n_k$ any of the four possibilities in Theorem \ref{flip} can be accomplished.
\end{proof}

\begin{proof}[Proof of Theorem \ref{onto}(a)] $H_{dim}(\lambda_2, d_{\lambda_3})$ can take any value in $[0,1]$. Pick $\alpha \in [0,1]$. Notice that $H_{dim}(\lambda_2,d_{\lambda_3}) =\liminf \frac{-\log(b_{k,2})}{\log(\frac{m_{k+1}}{n_{k+1}n_{k+2}b_{k,3}})}$. Choose ${m_{k+1}>(n_{k+1}b_{k,3})^k}$ and $n_{k+2}=\lfloor m_{k+1}^{\frac 1 {\alpha}} \rfloor$.
\end{proof}
\begin{proof}[Proof of Theorem \ref{onto}(b)] $H_{dim}(\lambda_3,d_{\lambda_2})$ can take any value in $[0,1]$. Pick $\alpha \in [0,1]$. Along an infinite subsequence of nonconsecutive $k$'s choose $m_k,n_k$ so that $n_k>b_{k-1,2}^k$ and $m_k=\lfloor n_k^{\frac 1 {\alpha}} \rfloor$. Choose $n_{k+1}=2m_{k+1}$ and $m_{k+1}=3n_{k+1}$. Notice that $\log_{\lambda_2(I_3^{(k)})}\lambda_3(I_3^{(k)})$ is proportional to $\frac{\log (\frac 1 {b_{k,3}})}{\log(\frac 1{b_{k,2}})+\log(\frac{n_{k+1}}{m_{k+1}})}$. By our assumptions we have that this is between $\alpha$ and $\alpha+\frac 2 k$. At the other $k$'s choose $m_k$ and $n_k$ to be the minimal allowed by Keane's construction (so $n_{k+1}=2m_{k} $ and $m_{k+1}=3n_{k+1}$).
\end{proof}

\subsection{Large sets of generic points}
The result of this section is Theorem \ref{generic} that the $\lambda_3$-generic points can be the complement of a set of Hausdorff dimension 0. This states that all but a tiny set of points behave typically for $\lambda_3$ at all times. Theorem \ref{generic} holds in particular when $m_k=3n_k$ and $n_{k+1}=b_{k,2}^k$.
\begin{Bob} Let $t_k(x)= \min\{n \geq 0:T^n(x) \in O(I_1^{(k)})\}$.
\end{Bob}
\begin{prop}\label{L3 gen} If $x \in \underset{n=1}{\overset{\infty}{\cup}} \underset{k=n}{\overset{\infty}{\cap}} O(I_3^{(k)})$ and $\underset{k \to \infty}{\lim} \frac{b_{k,1}}{t_k(x)}=0$ then $x$ is $\lambda_3 $ generic.
\end{prop}
\begin{proof} In the proof of Theorem 7 \cite{nonue}, Keane shows that $\underset{k=1}{\overset{\infty}{\prod}} \bar{A}_{m_k,n_k}e_3$ and $\underset{k=1}{\overset{\infty}{\prod}} \bar{A}_{m_k,n_k}e_4$ converge to $\lambda_3$. Therefore under the conditions of the hypothesis 
 $x$ is generic for $\lambda_3$. To see this, consider $b_{k-1,3}<s<b_{k,3}$. $x$ travels through $O(I_3^{(k-1)})$ $a$ times ($a=\frac{t_k}{b_{k-1,3}}-1$) then through $O(I_4^{(k-1)})$ then through $O(I_1^{(k-1)})$ then it lands back in $O(I_3^{(k-1)})$). By our assumption on $t_k$ the landing in $O(I_3^{(k-1)})$ eventually always dominates, so $x$ is $\lambda_3$-generic.
\end{proof}
\begin{prop} Under appropriate assumptions, the set of points in $\underset{n=1}{\overset{\infty}{\cup}} \underset{k=n}{\overset{\infty}{\cap}} O(I_3^{(k)})$ not satisfying the hypothesis of Proposition \ref{L3 gen} is a set of Hausdorff dimension 0.
\end{prop}
\begin{proof} $I_3^{(k)}$ travels up to $n_{k+1}$ times through $O(I_3^{(k)})$ before traveling through $O(I_4^{(k)})$ and then $O(I_1^{(k)})$. Therefore, the proportion of each level of $O(I_3^{(k)})$ that have $\frac{b_{k,1}}{t_k}<\epsilon$ is $\frac{\epsilon}{b_{k,1}} n_{k+1}b_{k,3}$. There are $b_{k,3}$ levels in the Rokhlin tower. So if $n_{k+1}$ is chosen so that $(\frac{b_{k,1}}{n_{k+1}})^{\frac 1 k} b_{k,3}<\frac 1 k$ then the set of $x \in \underset{n=1}{\overset{\infty}{\cup}} \underset{k=n}{\overset{\infty}{\cap}} O(I_3^{(k)})$ such that $\underset{k \to \infty}{\limsup} \frac{b_{k,1}}{t_k(x)}>0$ has Hausdorff dimension 0.
\end{proof}
\begin{proof}[Proof of Theorem \ref{generic}] By Lemmas \ref{L3I2small}, \ref{L3I4small} and  \ref{L3I1small} and the independence of the choice of $n_{k+1}$ of the previous $n_i$ and $m_i$ (and therefore $b_{i,j}$ for $i\leq k, j \in \{1,2,3,4\}$) it is easy to see that we may have $O(I_1^{(k)}) \cup O(I_2^{(k)}) \cup O(I_4^{(k)})$ have Hausdorff dimension 0 by choosing $n_{k+1}$ large enough (or $n_{k+2}$ large enough relative to $m_{k+1}$ for $O(I_2^{(k)})$). The theorem follows with the previous proposition.
\end{proof}
\section{Concluding remarks}
The previous discussion can be repeated in another example of minimal but not uniquely ergodic IETs: those arising from skew products over rotations \cite{vskew1}. In this case the ergodic measures are symmetric. Therefore if there are two ergodic measures $\mu_1$ and $\mu_2$ then $H_{dim}(\mu_1,d_{\mu_2})=H_{dim}(\mu_2,d_{\mu_1})$. We suspect that for almost every $\alpha$ the Hausdorff dimension for any ergodic measure obtained in this way is 1. Briefly, one considers a small interval and examines how often any point must hit it and apply Frostman's Lemma. We suspect that for exceptional $\alpha$ with carefully chosen continued fraction expansion any Hausdorff dimension can be obtained. To establish upper bounds one truncates the sum which defines the skewing interval to provide obvious $\delta$-coverings for any fixed $\delta>0$. The lower bound comes from Frostman's Lemma. Following \cite{ba skew} one can skew over two intervals. For any fixed irrational $\alpha$ one can obtain any Hausdorff dimension by appropriate choice of the two skewing intervals. The arguments are similar to those above. We end with a question.
\begin{ques}
(Cornfeld) Can any residual set carry an ergodic measure for a minimal IET?
\end{ques}
\section{Acknowledgments}

I would like to thank M. Boshernitzan, T. Coulbois, and S. Semmes for helpful conversations. This work was supported in part by Rice University's Vigre Grant and a Tracy Thomas award.

{xxx}
\begin{appendix}\section{Two ergodic measures that approximate each other differently}
\begin{thm}\label{different approaches} There exists a minimal 4-IET with two ergodic measures, $\lambda_2$ and $\lambda_3$ such that for any $\epsilon>0$ we have $\underset{n \to \infty}{\liminf}\, n^{1-\epsilon}d(T^nx,y)=0$ for $\lambda_2 \times \lambda_3$ almost every $(x,y)$ and $\underset{n \to \infty}{\liminf}\, n^{\frac 1 2 + \epsilon}d(T^nx,y)=\infty$ for $\lambda_3 \times \lambda_2$ almost every $(x,y)$.
\end{thm}
This will be proved in two parts (the $\lambda_3 \times \lambda_2$ statement and the  $\lambda_2 \times \lambda_3$ statement) under the assumption that $$m_k=k^2n_k \text{ and }n_{k+1}=b_{k,2}^2$$ and the distance is $d_{\lambda_{2}+\lambda_3}$, a metric that evenly weights the measures. $B(x,r)$ denotes the ball about $x$ of radius $r$ with respect to the metric $d_{\lambda_2+\lambda_3}$.
\begin{rem} By a straightforward modification one could prove the above theorem with $\underset{n \to \infty}{\liminf}\, n^{c }d(T^nx,y)=\infty$ for $\lambda_3 \times \lambda_2$ almost every $(x,y)$ for any $c \in [0,1]$.
\end{rem}
\begin{prop}\label{L3 side} Under the assumptions, for any $\epsilon>0$ and $\lambda_3 \times \lambda_2$ almost every point $(x,y)$ we have $\underset{n \to \infty}{\liminf} \,  n^{\frac 1 2 +\epsilon} \, d(T^nx,y)=\infty$.
\end{prop}
Notice that because the metric is $d_{\lambda_2+\lambda_3}$ it suffices to consider the $\lambda_2$ measure.

We first show that some points poorly approximate a $\lambda_2$ typical point.
\begin{lem} \label{controlled nice} If $ x \in \underset{t=1}{\overset{\lfloor \frac 1 {k^2} n_{k+1} b_{k,3} \rfloor}{\cap}} T^{-t}(O(I_3^{(k)}))$ then

\begin{multline} \lefteqn{\lambda_2(\underset{t=\lfloor \frac 1 {(k-1)^2} n_{k} b_{k-1,3} \rfloor}{\overset{\lfloor \frac 1 {k^2} n_{k+1} b_{k,3} \rfloor}{\cup}}B(T^tx,(\frac c {t^{\alpha}})))}\\ \leq \lambda_2(O(I_3^{(k)}))+(b_{k-1,1}+b_{k-1,4}+b_{k-1,3})2c(\lfloor \frac 1 {(k-1)^2} n_{k} b_{k-1,3} \rfloor)^{-\alpha}.\end{multline}
\end{lem}
\begin{proof} By our assumption $x$ lies in $O(I_3^{(k)})$ for time described, therefore the measure of the set is at most the measure of a $(\lfloor \frac 1 {(k-1)^2} n_{k} b_{k-1,3} \rfloor)^{-\frac 1 2}$ neighborhood of $O(I_3^{(k)})$. The lemma follows from observing that $I_3^{(k)}$ travels $n_k$ times through $O(I_3^{(k-1)})$ once through $O(I_1^{(k-1)})$ and once through $O(I_4^{(k-1)})$. One then groups the levels $O(I_3^{(k)})$ by the $O(I_i^{(k-1)})$ that they lie in.
\end{proof}
Next we show that these points are $\lambda_3$ typical.
\begin{lem} \label{control most} $\lambda_3\left(\underset{r=1}{\overset{\infty}{\cup}} \underset{k=r}{\overset{\infty}{\cap}} (\underset{t=1}{\overset{\lfloor \frac 1 {k^2} n_{k+1} b_{k,3} \rfloor}{\cap}} T^{-t}(O(I_3^{(k)}))\right)=1$.
\end{lem}
\begin{proof} First, observe that by how $T|_{I^{(k)}}$ acts on $I_3^{(k)}$ we have $$\lambda_3 \left(\underset{t=1}{\overset{\lfloor \frac 1 {k^2} n_{k+1} b_{k,3} \rfloor}{\cap}} T^{-t}(O(I_3^{(k)}))\right)\geq (1- \frac 1 {k^2})\lambda_3( O(I_3^{(k)})).$$ Also 
\begin{multline}\lefteqn{\lambda_3(O(I_1^{(k)}) \cup O(I_2^{(k)}) \cup O(I_4^{(k)}))=} \\b_{k,1}\lambda_3(I_1^{(k)})+b_{k,2}\lambda_3(I_2^{(k)})+b_k,4\lambda_3(I_4^{(k)})< \\  b_{k,1}\frac{1}{n_{k+1}}+b_{k,2}\frac{2m_{k+1}}{n_{k+1}n_{k+2}}+b_{k,4}\frac{1}{n_{k+1}}.\end{multline}
 
 For the inequalities observe that $\lambda_3(I_i^{(k)})< \frac{\lambda_3(I_i^{(k)})}{\lambda_3(I^{(k)})}$ and appeal to Lemmas \ref{L3I1small}, \ref{L3I2small} and \ref{L3I4small}. By Lemma \ref{b2 bigger} and our assumption $\frac{b_{k,i}}{n_{k+1}}<\frac 1 {b_{k,2}}$. Therefore the proposition follows by the Borel-Cantelli Theorem with the observation that $\underset{k=1}{\overset{\infty}{\sum}}\frac 1 {k^2}+\frac {3}{b_{k,2}}$ converges.
\end{proof}

\begin{proof}[Proof of Proposition \ref{L3 side}] By our assumption on $n_k,m_k$ it follows that $$\underset{k=1}{\overset{\infty}{\sum}} \lambda_2(O(I_3^{(k)}))+2c(b_{k-1,1}+b_{k-1,4}+b_{k-1,3})2(\lfloor \frac 1 {k^2} n_{k} b_{k+1,3} \rfloor)^{\frac 1 2 +\epsilon}$$ converges. By the Borel-Cantelli Theorem it follows that for all such $x$ we have $\lambda_2(\LS B(T^ix, \frac c {i^{\alpha}}))=0$. By Fubini's Theorem it follows that for all such $x$ we have $\underset{n \to \infty}{\limsup}\, n^{\frac 1 2 + \epsilon} d(T^nx,y)=\infty$. By Lemma \ref{control most} the proposition follows.
\end{proof}
\begin{prop} \label{L2 side} Under our assumptions for any $\epsilon>0$ and $\lambda_2 \times \lambda_3$ almost every point $(x,y)$ we have $\underset{n \to \infty}{\liminf}\, n^{1-\epsilon} \, |T^nx-y|=0$.
\end{prop}
We first show that some points approximate the $\lambda_3$ typical point not too poorly.
\begin{lem} If $T^tx \in O(I_2^{(k)})$ for $t<b_{k,3}$ then $\lambda_3(\underset{i=1}{\overset{b_{k,2}}{\cup}}B(T^ix,\frac{1}{b_{k,2}}^{1-\epsilon}))>\frac 1 2$ for large enough $k$.
\end{lem}
\begin{proof} By the assumption of the hypothesis the set $\{x,Tx,...,T^{b_{k,2}}x\}$ is at least $\lambda_2(I^{(k)})+\lambda_3(I^{(k)})$ dense in $O(I_3^{(k-1)})$. 
(The hypothesis of the Lemma ensures that $\{x,Tx,...,T^{b_{k,2}}x\}$ has $n_{k}$ hits in each level of $O(I_3^{(k-1)})$ by examining $T|_{I^{(k-1)}}(x)$, $O(I_2^{(k)})$ is $\lambda_2(I^{(k)})+ \lambda_3(I^{(k)})$ dense in $O(I_3^{(k-1)})$.) 
 Notice that $\lambda_3(I^{(k)})< \frac 1 {b_{k,4}}$ because there are $b_{k,4}$ disjoint copies of $I^{(k)}$ in $I$. By our choice of $m_k$ and $n_k$, $b_{k,4}>b_{k,2}^{1-\epsilon}$ for all large enough $k$. Also $\lambda_2(I^{(k)})$ is proportional to $\frac 1 {b_{k,2}}$ by Lemma \ref{L2bigorbit}. Therefore, $\{x,Tx,...,T^{b_{k,2}}\}$ is $\frac 1 {b_{k,4}^{1+\epsilon}}$ dense for all large enough $k$ and the Lemma follows.
\end{proof}
We next show that these points are $\lambda_2$ significant.
\begin{lem} The set of points satisfying the hypothesis of the above Lemma has $\lambda_2$ measure at least $\frac 1 8$.
\end{lem}
\begin{proof} This follows from the fact that $\lambda_2(O(I_2^{(k)}))>\frac 1 4$ (Lemma \ref{L2bigorbit}) 
 and $1-\frac{b_{k,3}}{b_{k,2}}> \frac 1 2$ of these points satisfy the hypothesis of the lemma.
\end{proof}
\begin{proof}[Proof of Proposition \ref{L2 side}] The proof follows from Fubini's Theorem and ergodicity. \end{proof}
\end{appendix}
\end{document}